\documentclass[11pt,a4paper]{article}

\usepackage{mathtools}
\usepackage{amsthm,amsmath}
\usepackage{amssymb,MnSymbol}
\usepackage[svgnames]{xcolor}
\usepackage{tabto}
\usepackage{paralist}
\usepackage[
colorlinks=true,
linkcolor=MidnightBlue,
citecolor=MidnightBlue,
urlcolor=MidnightBlue,
pagebackref=true]{hyperref}
\usepackage[alphabetic]{amsrefs}
\usepackage{graphicx}
\usepackage{pinlabel}
\usepackage{enumitem}
\usepackage[norefs]{refcheck}
\usepackage[margin=3.3cm]{geometry}

\newcommand{\N}{\mathbb{N}}
\newcommand{\Z}{\mathbb{Z}}
\newcommand{\R}{\mathbb{R}}
\newcommand{\C}{\mathbb{C}}

\newcommand{\I}{\mathcal{I}}
\newcommand{\U}{\mathcal{U}}

\newcommand{\D}{\mathcal{D}}

\newcommand{\Ga}{\Gamma}

\newcommand{\al}{\alpha}
\newcommand{\be}{\beta}
\newcommand{\de}{\delta}

\newcommand{\ga}{\gamma}
\newcommand{\la}{\lambda}

\newcommand{\id}{\mathrm{id}}

\newcommand{\PSL}{\mathrm{PSL}}

\newcommand{\Hom}{\mathrm{Hom}}

\newcommand{\homM}{\mathrm{Homeo}(M)}
\newcommand{\Homeo}{\mathrm{Homeo}}

\newcommand{\Ray}{\mathrm{Ray}}
\newcommand{\Code}{\operatorname{Code}}
\newcommand{\diam}{\operatorname{diam}}
\DeclareMathOperator{\cl}{cl}
\numberwithin{equation}{section}

\theoremstyle{plain}
\newtheorem{theorem}{Theorem}[section]
\newtheorem{corollary}[theorem]{Corollary}

\newtheorem{lemma}[theorem]{Lemma}
\newtheorem{proposition}[theorem]{Proposition}

\newtheorem*{claim*}{Claim}
\newtheorem*{theorem*}{Theorem}
\newtheorem*{proposition*}{Proposition}

\theoremstyle{definition}
\newtheorem{definition}[theorem]{Definition}
\newtheorem{example}[theorem]{Example}

\theoremstyle{remark}

\title{Local $C^0$-semi-rigidity of meandering-hyperbolic actions}
\author{Sungwoon Kim}

\date{}
\begin{document}

\maketitle

\begin{abstract}
Meandering hyperbolicity, introduced by Kapovich, Kim, and Lee, extends classical hyperbolicity beyond the setting of word-hyperbolic groups. In this paper, we prove that every meandering-hyperbolic action is locally semi-rigid in the $C^0$ topology. This extends the previously established stability theory in the Lipschitz $C^0$ topology to the full $C^0$ topology. Consequently, we recover the $C^0$-local semi-rigidity of both boundary actions of word-hyperbolic groups and cocompact lattices in semisimple Lie groups from a single dynamical principle.
\end{abstract}


\section{Introduction}

A central problem in the theory of discrete groups and geometric structures is to determine which geometric properties of representations are stable under deformation. Such questions arise naturally in representation varieties, which contain many geometrically significant representations, including holonomy representations of geometric structures. The study of the stability of representations has become a central theme connecting geometry, topology, and dynamical systems.

A classical example is provided by convex cocompact Kleinian groups. Thurston \cite{Thu79} introduced the notion of convex cocompact Kleinian groups, and Sullivan \cite{Sul} showed that convex cocompactness is essentially
the only stable geometric property among representations into
$\PSL(2,\C)$. His work established a fundamental connection between stability and the dynamics of boundary actions.

The notion of convex cocompactness was later generalized to higher-rank semisimple Lie groups through the theory of Anosov representations. Labourie \cite{Lab} first introduced Anosov representations in his study of Hitchin components, and Guichard and Wienhard \cite{GW} extended the notion to arbitrary word-hyperbolic groups. Subsequently, Kapovich, Leeb, and Porti \cite{KLP14,KLP17,KLP18} established several equivalent characterizations in geometric,
dynamical, and coarse-geometric terms. Anosov representations form an open and structurally stable class of representations and have become central objects in higher Teichmüller theory.
Despite these developments, Anosov representations are intrinsically tied to word-hyperbolic groups. This naturally raises two fundamental questions.

\begin{enumerate}
\item Why do stable representations arise primarily from word-hyperbolic groups?

\item Which dynamical properties are fundamental to the stability phenomenon?
\end{enumerate}

The first question has led to numerous attempts to develop stable notions beyond word-hyperbolic groups. For example, Danciger--Guéritaud--Kassel \cite{DGK} introduced projective convex cocompactness, a notion that includes
examples arising from non-word-hyperbolic groups. Nevertheless, most existing theories remain closely tied to hyperbolic-type phenomena and provide only isolated examples rather than a conceptual explanation of which dynamical properties of non-word-hyperbolic groups underlie stability.

The second question was largely clarified by Sullivan's dynamical viewpoint. Rather than viewing convex cocompactness as a consequence of hyperbolic geometry, Sullivan showed that its stability originates from the expanding--contracting dynamics and hyperbolicity of the induced boundary action. From this perspective, he introduced the notion of hyperbolic actions and proved that such actions are stable in the $C^1$ topology. Furthermore, he showed that convex cocompact group actions on the Riemann sphere are hyperbolic actions and therefore remain stable under perturbations within the much larger group of diffeomorphisms of the Riemann sphere. This demonstrated that the mechanism underlying stability is dynamical
rather than purely geometric.

More recently, Kapovich, Kim, and Lee \cite{KKL} extended Sullivan's dynamical viewpoint by introducing the notion of \emph{meandering hyperbolicity} that is a purely dynamical condition that extends
classical hyperbolicity beyond the setting of word-hyperbolic groups. They proved that the boundary actions of both word-hyperbolic groups and cocompact lattices in semisimple Lie groups are meandering hyperbolic, thereby providing a common dynamical principle for these two classes of actions. They established structural stability for meandering-hyperbolic actions in the Lipschitz $C^0$ topology.

Around the same time, Mann--Manning--Weisman \cite{MMW} proved a local
$C^0$-semi-rigidity theorem for boundary actions of word-hyperbolic
groups by methods from geometric group theory. Independently,
Connell--Islam--Nguyen--Spatzier \cite{CINS} established the analogous
result for boundary actions of cocompact lattices in semisimple Lie
groups using ergodic-theoretic methods. Although meandering hyperbolicity provides a common framework for these
two classes of actions, the two results were obtained independently and
by substantially different methods.

In this paper, we prove that every meandering-hyperbolic action is locally
semi-rigid in the full $C^0$ topology. Consequently, the local $C^0$-semi-rigidity results of
Mann--Manning--Weisman and Connell--Islam--Nguyen--Spatzier follow from
a single dynamical theorem. Thus, our result extends the structural stability theory of Kapovich--Kim--Lee from the Lipschitz $C^0$ topology to the full $C^0$ topology and shows that meandering hyperbolicity captures the essential dynamical mechanism underlying stability.

Our main theorem is the following.

\begin{theorem}\label{thm:main}
Let $M$ be a compact metric space, let $\Gamma<\Homeo(M)$ be finitely
generated, and let $
\iota : \Gamma\longrightarrow\Homeo(M)$
denote the inclusion homomorphism. Suppose that the $\Gamma$-action on $M$ is meandering hyperbolic. Then every action
$
\rho : \Gamma\longrightarrow\Homeo(M)
$
that is sufficiently close to $\iota$ in the $C^0$ topology is
semi-conjugate to $\iota$.
\end{theorem}

Our proof uses only dynamical and topological arguments.
While inspired by the ideas of Mann--Manning--Weisman
\cite{MMW}, our argument is carried out entirely within the framework
of meandering hyperbolicity.
Starting with an expanding datum, we establish uniform contraction
properties that remain stable under sufficiently small $C^0$
perturbations.
These estimates allow us to construct the desired semi-conjugacy from
perturbed domain sequences, and meandering hyperbolicity ensures that
the construction is independent of the chosen code.

Since both boundary actions of word-hyperbolic groups  and cocompact lattices in semisimple Lie groups are
meandering hyperbolic \cite{KKL}, the theorem immediately yields the
following applications.

\begin{corollary}
The action of a word-hyperbolic group on its Gromov boundary is locally
semi-rigid in the $C^0$ topology. Likewise, if $\Gamma$ is a cocompact
lattice in a semisimple Lie group $G$ and $P<G$ is a parabolic subgroup,
then the action of $\Gamma$ on $G/P$ is locally semi-rigid in the
$C^0$ topology.
\end{corollary}

The results of this paper provide further evidence that meandering hyperbolicity furnishes a natural dynamical framework for studying stability phenomena in representation varieties. We expect this viewpoint to provide a useful framework for studying
stable representations beyond the classical hyperbolic setting and to
lead to further applications in higher-rank geometry and geometric group
theory.

The paper is organized as follows.
In Section~2, we review the background on hyperbolic spaces, Cayley graphs,
and expanding actions.
In Section~3, we establish the uniform contraction properties of expanding actions that form the dynamical foundation of the paper.
In Section~4, we prove that these contraction properties persist under sufficiently small $C^0$
-perturbations and develops the perturbation theory needed for the proof of the main theorem.
Finally, in Section~5, we combine these results with meandering hyperbolicity
to prove the local $C^0$-semi-rigidity theorem and derive the corresponding
applications to boundary actions of word-hyperbolic groups and cocompact
lattices in semisimple Lie groups.

\section{Preliminaries}

In this section, we collect the background material used throughout the
paper. We briefly review hyperbolic spaces, Cayley graphs, expanding
actions, and meandering hyperbolicity. Throughout the paper, $M$ denotes a compact metric space,
$\Homeo(M)$ denotes its group of homeomorphisms, and
$\Gamma<\Homeo(M)$ denotes a finitely generated subgroup.

\subsection{Hyperbolic spaces and Cayley graphs}

We recall Gromov hyperbolicity and Cayley
graphs. These notions provide the geometric setting for
word-hyperbolic groups.

Let $M$ be a metric space. We denote by $B(x,r)$ the open ball of radius $r$ centered at $x$, and
by $\cl(A)$ the closure of a subset $A\subset M$.
A metric space $M$ is said to be \emph{proper} if every closed ball is compact. If every pair of points can be joined by a geodesic segment, $M$ is called a \emph{geodesic space}. 

\begin{definition}[Quasi-geodesic]
Let $I$ be an interval of $\R$ (or its intersection with $\Z$) and $(X,d)$ a metric space. A map $c : I\to X$ is a $(B,C)$-quasi-geodesic, where
$B\geq 1$ and $C\geq 0$, if for all $t,t'\in I$,
\[
\frac{1}{B}|t-t'|-C\le d(c(t),c(t'))\le B|t-t'|+C.
\]
\end{definition}

A quasi-geodesic is the coarse analogue of a geodesic. Gromov \cite{Gro} introduced $\de$-hyperbolic spaces as a metric analogue of negatively curved manifolds to capture the essential large-scale features of hyperbolic geometry and to study groups from a geometric perspective.

\begin{definition}[Hyperbolic space]
Fix a basepoint $o\in M$. For $x,y\in M$, define the Gromov product by
\[
(x\mid y)_o
 =\frac12\bigl(d(x,o)+d(y,o)-d(x,y)\bigr).
\]
The space $M$ is \emph{$\delta$-hyperbolic} if
\[
(x\mid z)_o
 \geq \min\{(x\mid y)_o,(y\mid z)_o\}-\delta
\]
for all $x,y,z\in M$.
A geodesic space is said to be \emph{hyperbolic} if it is $\de$-hyperbolic for some $\de\ge 0$.
\end{definition}

Although Gromov originally defined hyperbolicity using the Gromov product, it was subsequently established that, in the setting of geodesic metric spaces, this definition is equivalent to the $\de$-thin triangle condition. More precisely, $M$ is $\de$-hyperbolic if and only if for every geodesic triangle in $M$, each side of the triangle is contained in the closed $\de$-neighborhood of the union of the other two sides. 

We now recall Cayley graphs and word metrics.
Let $\Gamma$ be a finitely generated group with a finite generating set
$S$. The word metric on $\Gamma$ associated to $S$ is defined by
\[
d_S(g,h)=|g^{-1}h|_S,
\]
where $|\gamma|_S$ denotes the word length of $\gamma\in\Gamma$,
namely,
\[
|\gamma|_S
 =\min\{n\geq0  :  \gamma=s_1\cdots s_n,\ 
 s_i\in S\cup S^{-1}\}.
\]
Equivalently, $d_S(g,h)$ is the length of the shortest word in the
alphabet $S\cup S^{-1}$ representing $g^{-1}h$. Assigning length $1$ to each edge turns the Cayley graph into a geodesic metric space. A finitely generated group $\Gamma$ is called \emph{word-hyperbolic} if its Cayley graph is hyperbolic. Word hyperbolicity is independent of the choice of the finite generating set, since any two Cayley graphs of $\Gamma$ are quasi-isometric.


\subsection{Expanding actions}

We recall the notion of an expanding action introduced by
Kapovich, Kim, and Lee \cite{KKL}. The main objects associated with an expanding action are
nested domain sequences, whose uniform contraction properties will
play a central role throughout the paper.

We first introduce the local notion of an expanding homeomorphism.
Let $(M,d)$ be a metric space and let $f\in\Homeo(M)$.
Denote by $f[A]$ the image of $A\subset M$ under $f$. Given $\la>1$ and $U\subset M$, we say that $f$ is \emph{$(\la,U)$-expanding} (or \emph{$\la$-expanding on $U$}) if
\[
d(f(x),f(y))\ge\la d(x,y)
\]
for all $x,y\in U$. In this case, we also say that $U$ is a \emph{$(\la,f)$-expanding domain}. 

\begin{definition}[Expanding homeomorphism]
Let $\lambda>1$, $\delta>0$, and $U\subset M$.
A homeomorphism $f\in\Homeo(M)$ is
$(\lambda,U;\delta)$-expanding if
\[
B(f(x),\la r)\subset f[B(x,r)]\ \textup{ whenever } B(x,r)\subset U \textup{ and }0<r\le\de.
\]
\end{definition}

It follows immediately from the definition that if $f$ is $(\la,U;\de)$-expanding, it is also $(\la,U;\eta)$-expanding for any $0<\eta\le\de$. 
The dual notion of an expanding homeomorphism is defined as follows.
\begin{definition}[Contracting homeomorphism]
Let $\lambda>1$, $\delta>0$, and $U\subset M$.
A homeomorphism $f\in\Homeo(M)$ is
$(\lambda^{-1},U;\delta)$-contracting if
\[
f[B(x,r)]\subset B(f(x),\la^{-1}r) \textup{ whenever } B(x,r)\subset U \textup{ and }0<r\le\de.
\]
\end{definition}

By definition, the inverse of an expanding homeomorphism satisfies the following contraction property.

\begin{lemma}\label{lem:con}
Suppose that $f$ is $(\lambda,U;\delta)$-expanding. If
$B(f^{-1}(y),\lambda^{-1} r)\subset U$
for some $0<r\le \lambda\delta$, then
\[
f^{-1}[B(y,r)]
\subset
B\left(f^{-1}(y),\lambda^{-1}r\right).
\]
\end{lemma}
\begin{proof}
This lemma is immediate from the definition.
\end{proof}

\begin{definition}
A homeomorphism $f\in \homM$ is said to be \emph{expanding} at $x\in M$ if there exists a neighborhood $U$ of $x$ such that $f$ is $(\la,U;\de)$-expanding for some $\la>1$ and $\de>0$.
We say that $\Gamma<\Homeo(M)$ is expanding on $M$ if, for every
$x\in M$, there exists $\gamma\in\Gamma$ such that $\gamma^{-1}$
is expanding at $x$.
\end{definition}

Compactness yields the following finite description of an expanding
action.

\begin{proposition}[Existence of expanding data]
Let $M$ be a compact metric space, and let
$\Gamma<\Homeo(M)$ be a finitely generated group that is expanding on
$M$. Then there exist an open cover
$
\mathcal U=\{U_1,\ldots,U_n\}
$
of $M$, a finite symmetric generating set
$
S=\{s_1,\ldots,s_n\}
$
of $\Gamma$, and constants $\lambda>1$ and $\delta>0$ such that
$s_i^{-1}$ is $(\lambda,U_i;\delta)$-expanding for $1\leq i\leq n$.
\end{proposition}

\begin{proof}
For each $x\in M$, choose $\gamma_x\in\Gamma$, constants
$\lambda_x>1$ and $\delta_x>0$, and a neighborhood $U_x$ of $x$
such that $\gamma_x^{-1}$ is
$(\lambda_x,U_x;\delta_x)$-expanding.
By compactness, the open cover $\{U_x\}_{x\in M}$ admits a finite
subcover
$
U_{x_1},\ldots,U_{x_n}.
$
Set
\[
\lambda=\min_{1\le i\le n}\lambda_{x_i}>1,
\qquad
\delta=\min_{1\le i\le n}\delta_{x_i}>0,
\]
and, for $1\le i\le n$, set
\[
U_i=U_{x_i},
\qquad
s_i=\gamma_{x_i}.
\]
Then $\mathcal U=\{U_1,\ldots,U_n\}$ is an open cover of $M$, and
$s_i^{-1}$ is $(\lambda,U_i;\delta)$-expanding for
 $1\leq i\leq n$.

The finite set $S=\{s_1,\ldots,s_n\}$ need not generate $\Gamma$. Since $\Gamma$ is finitely generated, $S$ is contained in a finite
symmetric generating set $S'$ of $\Gamma$. Write
\[
S'=S\cup\{s_{n+1},\ldots,s_m\}.
\]
For $i=n+1,\ldots,m$, let $U_i$ be the $(\la,s_i^{-1})$-expanding domain, which may be empty. Then
\[
\mathcal U'=\{U_1,\ldots,U_m\}
\]
is still an open cover of $M$, and the expansion condition for the
additional generators holds. Replacing $\mathcal U$ and $S$
by $\mathcal U'$ and $S'$, respectively, completes the proof.
\end{proof}

Motivated by the preceding proposition, we introduce the following notion.

\begin{definition}[Expanding datum]
An \emph{expanding datum} for the $\Gamma$-action on $M$ is a tuple
$\mathcal D=(I,\mathcal U,S,\lambda,\delta)$ satisfying the following
conditions for some $n\in\mathbb N$:
\begin{itemize}
\item[\textup{(i)}] $I=\{1,\ldots,n\}$,
\item[\textup{(ii)}] $\mathcal U=\{U_1,\ldots, U_n\}$ is an open cover of $M$,
\item[\textup{(iii)}] $S=\{s_1,\ldots, s_n\} \subset \Gamma$ is a symmetric generating set of $\Gamma$, i.e., $S=S^{-1}$,
\item[\textup{(iv)}] $\la>1$, and $\de>0$ is a Lebesgue number for $\U$,
\item[\textup{(v)}]  $s_i^{-1}$ is $(\la,U_i;\de)$-expanding for all $i\in I$.
\end{itemize}
\end{definition}

The preceding proposition shows that every expanding action on a compact
metric space admits an expanding datum. Moreover, if
$(I,\mathcal U,S,\lambda,\delta)$ is an expanding datum, then so is
$(I,\mathcal U,S,\lambda,\eta)$ for every $0<\eta\leq\delta$.

Fix an expanding datum
$\mathcal D=(I,\mathcal U,S,\lambda,\delta)$ and
$\eta\in(0,\delta]$. Given $x\in M$, set $x_0=x$ and choose
$\alpha(0)\in I$.
Define recursively
\[
x_{k+1}=s_{\alpha(k)}^{-1}(x_k).
\]
Since $\delta$ is a Lebesgue number for $\mathcal U$ and
$\eta\leq\delta$, there exists $\alpha(k+1)\in I$ such that
\[
B(x_{k+1},\eta)\subset U_{\alpha(k+1)}.
\]
Continuing inductively,
we obtain a sequence
$\alpha:\mathbb N\to I$,
called a
$(\mathcal D,\eta)$-code
for
$x$.

\begin{definition}[Code] Let $\mathcal D=(I,\mathcal U,S,\lambda,\delta)$ be an expanding datum,
let $0<\eta\leq\delta$, and let $x\in M$.
A sequence $\al : \N \to I$ is called a \emph{$(\D,\eta)$-code} for $x$ if 
\begin{itemize}
\item[\textup{(i)}] $x_0^\al=x$,
\item[\textup{(ii)}] $x_{k+1}^\al=s_{\alpha(k)}^{-1} (x_k^\al)$ for all $k\ge 0$,
\item[\textup{(iii)}] $B(x_k^\al,\eta)\subset U_{\alpha(k)}$ for all $k\geq 1$.
\end{itemize}
A $(\mathcal D,\eta)$-code $\alpha$ is called \emph{special} if, in addition,
$
B(x^\alpha_0,\eta)\subset U_{\alpha(0)}.
$
\end{definition}

For each $(\D,\eta)$-code $\al$ for $x$, define its ray $r^\al :\N\to \Gamma$ by 
\[ r^\alpha_k=s_{\alpha(0)}\cdots s_{\alpha(k)}\]
and its domain sequence by 
\[ D^\alpha_k=r^\al_k [{B(x_{k+1}^\al,\eta)}].\]
By construction,
$r_k^\alpha(x_{k+1}^\alpha)=x$, and therefore $x\in D_k^\alpha$ for every $k\ge0$.

The sequence $(D_k^\alpha)_{k\ge0}$ is called the
\emph{$\alpha$-domain sequence}. These domain sequences are the main technical objects used to construct
the semi-conjugacy. In Section~3, we show that they are nested and shrink
uniformly.

For $x\in M$, let
\[
\Code_x(\mathcal D,\eta)
\quad\text{and}\quad
\Ray_x(\mathcal D,\eta)
\]
denote the sets of all
$(\mathcal D,\eta)$-codes
and
$(\mathcal D,\eta)$-rays
for
$x$,
respectively.

An expanding action generally admits many different expanding data.
For later use, we introduce the following refinement relation.

\begin{definition}[Refinement]\label{def:refine}
Suppose that $\Gamma<\Homeo(M)$ is expanding on $M$. Let
$\mathcal D=(I,\mathcal U,S,\lambda,\delta)$ and 
$\mathcal D'=(I',\mathcal U',S',\lambda',\delta')$
be expanding data for the
$\Gamma$-action.
We write $\mathcal D\prec\mathcal D'$ and say that
$\mathcal D'$
is a
\emph{refinement}
of
$\mathcal D$
if

\begin{enumerate}[label=(\roman*),nosep]
\item
$I\subset I'$;

\item
$U_i\subset U_i'$
for every
$i\in I$;

\item
$s_i=s_i'$
for every
$i\in I$;

\item
$0<\delta\le\delta'$;

\item
$\lambda\ge\lambda'>1$.
\end{enumerate}
\end{definition}

It follows directly from the definitions that if
$\mathcal D\prec\mathcal D'$,
then
\[
\Code_x(\mathcal D,\eta)
\subset
\Code_x(\mathcal D',\eta)
\]
and
\[
\Ray_x(\mathcal D,\eta)
\subset
\Ray_x(\mathcal D',\eta)
\]
for every
$x\in M$ and every $\eta\in (0,\de]$. Thus, a refinement enlarges the admissible expanding domains
while allowing the expansion constant to decrease and the admissible
radius to increase.

\subsection{Meandering hyperbolicity}

We recall the notion of meandering hyperbolicity introduced by
Kapovich, Kim, and Lee \cite{KKL}. It is formulated in terms of the
coarse behavior of rays arising from codes associated with an expanding
datum. The key notion is the following equivalence relation.

\begin{definition}[$L$-equivalence]\label{def}
Suppose that $\Gamma<\homM$ is expanding on $M$ with an expanding
datum $\D=(I,\U,S,\la,\de)$,
and let $0<\eta\le\de$.

\begin{enumerate}[label=(\alph*),nosep,leftmargin=*]
\item
Define a relation $\approx_L$ on
$\Ray_x(\D,\eta)$ as follows. For
$r^\al,r^\be\in\Ray_x(\D,\eta)$, we write
\[
r^\al\approx_L r^\be
\]
if there exist infinite subsets $P,Q\subset\mathbb N$ for which the
subsets
\[
r^\alpha(P)=\{r^\alpha_k:k\in P\},
\qquad
r^\beta(Q)=\{r^\beta_j:j\in Q\}
\]
have Hausdorff distance at most $L$ in $(\Gamma,d_S)$.

The \emph{$L$-equivalence relation}, denoted by $\sim_L$, is the
equivalence relation on $\Ray_x(\D,\eta)$ generated by $\approx_L$.
Thus, we write
\[
r^\al\sim_L r^\be
\]
and say that $r^\al$ and $r^\be$ are \emph{$L$-equivalent} if there
exists a finite sequence of rays
\[
r^\al=r^{\ga_0},r^{\ga_1},\ldots,r^{\ga_m}=r^\be
\]
in $\Ray_x(\D,\eta)$ such that
\[
r^{\ga_0}\approx_L r^{\ga_1}
\approx_L\cdots\approx_L r^{\ga_m}.
\]

\item
Two $(\D,\eta)$-rays $r^\al$ and $r^\be$ are said to
\emph{$L$-fellow-travel} if their images
$r^\al(\N)$ and $r^\be(\N)$ have Hausdorff distance at most $L$
in $(\Gamma,d_{S})$.
\end{enumerate}
\end{definition}

We can now state the definition of meandering hyperbolicity.

\begin{definition}[Meandering hyperbolicity]\label{def}
Suppose that $\Gamma<\homM$ is expanding on $M$ with an expanding
datum $\D=(I,\U, S,\la,\de)$.

\begin{enumerate}[label=(\alph*),nosep,leftmargin=*]
\item
The $\Gamma$-action on $M$ is said to be
\emph{meandering hyperbolic} if there exist a refinement
$\D'$ of $\D$ and a constant $L\geq0$ such that, for every
$x\in M$ and every $\eta\in (0,\de]$, all rays in $\Ray_x(\D,\eta)$ are $L$-equivalent when regarded
as rays in $\Ray_x(\D',\eta)$. The triple
\[
(\D\prec\D';L)
\]
is called a \emph{meandering-hyperbolicity datum} for the
$\Gamma$-action on $M$.

\item
The $\Gamma$-action on $M$ is said to be \emph{hyperbolic} if there
exists a constant $L\ge0$ such that, for every $x\in M$, any two rays in $\Ray_x(\D,\delta)$ $L$-fellow-travel.
\end{enumerate}

\end{definition}

Replacing the fellow-traveling condition by $L$-equivalence is the key
generalization in the definition of meandering hyperbolicity. In the
proof of the main theorem, this weaker relation is nevertheless sufficient to ensure that the
limit associated with a point is independent of the chosen code.


\section{Uniform Contraction for Expanding Actions}
\label{sec:ucea}

Throughout this section, let
$\mathcal D=(I,\mathcal U,S,\lambda,\delta)$
be a fixed expanding datum for an expanding $\Ga$-action on $M$.

We establish uniform contraction estimates for the domain sequences
associated with $\D$. The principal result, Proposition~\ref{prop:ucp}, shows that
these estimates are independent of the initial point and the chosen
code; this uniformity is essential for the proof of local
$C^0$-semi-rigidity.

\subsection{Stability of expanding homeomorphisms}

We first prove that the local expanding property is stable under
small $C^0$ perturbations.
Although the expansion factor and admissible radius may have to be
slightly reduced, the expanding property itself persists uniformly.
This will be used repeatedly in Section 4.

Equip $\Homeo(M)$ with the uniform metric

\[
d_\infty(f,g)
=
\sup_{x\in M}d(f(x),g(x))
+
\sup_{x\in M}d(f^{-1}(x),g^{-1}(x)),
\]
which induces the usual $C^0$ topology.
The representation space
$\Hom(\Gamma,\Homeo(M))$
is endowed with the topology of pointwise convergence on a finite
generating set.

The following lemma is the basic perturbation estimate used throughout
the remainder of the paper.

\begin{lemma}\label{lem:sep0}
Let $f$ be a homeomorphism of a compact metric space $M$ and let
$x\in M$.
Suppose that there exist constants
$\lambda>1$
and
$\delta>0$
such that
\[
f[B(x,r)]
\supset
B(f(x),\lambda r)
\]
for every
$0<r\le\delta$.
For $\de_0\in (0,\de)$ and $\la_0\in (1,\la)$, there exist a neighborhood
$\mathcal V(\de_0,\la_0)$
of
$f$
in
$\Homeo(M)$
and a neighborhood
$U_0$
of
$x$
in
$M$
such that
\[
g[B(y,r)]
\supset
B(f(y),\lambda_0r)
\]
for every
$g\in\mathcal V(\de_0,\la_0)$,
every
$y\in U_0$,
and every
$r\in[\delta_0,\delta]$.
\end{lemma}

\begin{proof}
Assume the contrary. Then, for any neighborhood
$\mathcal V$ of $f$ in $\homM$ and any neighborhood $U$ of $x$ in $M$, there exist
$g\in\mathcal V, \ y\in U$, and $r\in[\de_0,\delta]$
such that
\[
g\bigl[B(y,r)\bigr]\not\supset B\bigl(f(y),\la_0 r\bigr).
\]

Thus, there exist a sequence $(g_n)$ in $\homM$ converging to $f$, a sequence $(y_n)$ in $M$ converging to $x$, a sequence $(r_n)$ in $[\de_0,\de]$ and a sequence $(z_n)$ in $M$ such that
\[
z_n\in B\bigl(f(y_n),\lambda_0 r_n\bigr)
\text{ but }
z_n\notin g_n\bigl[B(y_n,r_n)\bigr].
\]

Since both $M$ and $[\de_0,\delta]$ are compact, after passing to a subsequence we may assume that
$z_n$ converges to $z$ for some $z\in M$ and $r_n$ converges to $r_0$ for some $r_0\in [\de_0,\de]$.
Since $y_n$ converges to $x$ and $f$ is continuous, $f(y_n)$ converges to $f(x)$ by the continuity of $f$.
Thus
\[
d(z,f(x))
=
\lim_{n\to\infty} d\bigl(z_n,f(y_n)\bigr)
\le
\lim_{n\to\infty} \lambda_0 r_n
=
\lambda_0 r_0.
\]
From $\lambda_0<\lambda$, we obtain
\[
d(z,f(x))\leq \la_0 r_0< \lambda r_0,
\]
and thus
\[
z\in B\bigl(f(x),\lambda r_0\bigr).
\]

By the hypothesis that $f\bigl[B(x,r_0)\bigr]\supset B\bigl(f(x),\lambda r_0\bigr)$,
there exists $w\in B(x,r_0)$ such that $f(w)=z$.
Since $g_n$ converges to $f$ in $\homM$ and $M$ is compact, $g_n^{-1}$ uniformly converges to $f^{-1}$, which implies that 
$g_n^{-1}(z_n)$ converges to $f^{-1}(z)=w$.
Together with the fact that $y_n$ converges to $x$, this yields
\[
\lim_{n\to \infty} d\bigl(g_n^{-1}(z_n),y_n\bigr) = d(w,x)<r_0.
\]

Since $r_n$ converges to $r_0$ and the limit of $d\bigl(g_n^{-1}(z_n),y_n\bigr)$ is less than $r_0$, there exists a number $N$ such that 
\[
d\bigl(g_N^{-1}(z_N),y_N\bigr)<r_N,
\]
which implies that
\[
z_N
=
g_N\bigl(g_N^{-1}(z_N)\bigr)
\in
g_N\bigl[B(y_N,r_N)\bigr].
\]
This contradicts the choice of $z_N$ and completes the proof.
\end{proof}

The lower bound on the radius in Lemma~\ref{lem:sep0} is essential.
The following example shows that the conclusion generally fails if
arbitrarily small radii are allowed.

\begin{example}
Let
$f:\mathbb R\to\mathbb R$
be defined by
$f(x)=2x$.
Then
$f$
is
$(2,\mathbb R;\delta)$-expanding
for every
$\delta>0$.

For each
$n\in\mathbb N$,
define a homeomorphism
$f_n:\mathbb R\to\mathbb R$
by
\[
f_n(x)=
\begin{cases}
2x, & x\leq0,\\[2mm]
\dfrac{x}{2}, & 0\leq x\leq \dfrac1n,\\[2mm]
2\left(x-\dfrac1n\right)+\dfrac1{2n},
  & x\geq\dfrac1n.
\end{cases}
\]

A direct calculation shows that $f_n$ and $f_n^{-1}$ uniformly converges to $f$ and
$f^{-1}$, respectively.
However,
\[
f_n\!\left[B\!\left(0,\frac1n\right)\right]
=
\left(-\frac1n,\frac1{2n}\right)
\not\supset
\left(-\frac1n,\frac1n\right)
=
B\!\left(f(0),\frac1n\right).
\]

Consequently, for every $1<\lambda_0<2$, there is no neighborhood
$\mathcal V$ of $f$ such that
\[
g[B(0,r)]\supset B(f(0),\lambda_0r)
\]
for every $g\in \mathcal V$ and every $0<r\leq\delta$.
Thus the positive lower bound
$r\ge\delta_0$
in Lemma~\ref{lem:sep0}
is essential.
\end{example}

Lemma \ref{lem:sep0} is local because the neighborhoods depend on the base point.
Compactness yields the following uniform version on an expanding
domain.

\begin{proposition}[Uniform perturbation of expanding homeomorphisms]\label{prop:deformp}
Let
$f$
be a
$(\lambda,U;\delta)$-expanding
homeomorphism
of a compact metric space
$M$.
Let
$0<\delta_0<\delta$
and
$1<\lambda_0<\lambda$.
Then there exists a neighborhood $\mathcal V(\de_0,\la_0)$ of $f$ in $\Homeo(M)$ such that
\[
g[B(x,r)]\supset B(f(x),\lambda_0r)
\]
for every $g\in \mathcal V(\de_0,\la_0)$ whenever
$B(x,r)\subset U$ and $\delta_0\leq r\leq\delta$.
\end{proposition}

\begin{proof}
For each
$x\in U$,
define
\[
\delta_x
=
\sup
\{
r>0 :
B(x,r)\subset U
\} \text{ and }\bar \de_x
=
\min\left\{
\delta,\,
\de_x
\right\}.
\]
Let
$
K
=
\{
x\in U:
\delta_x\ge\delta_0
\}=\{x\in M :  d(x,M\setminus U)\geq\delta_0\}.
$
Since the function $x \mapsto d(x,M\setminus U)$ is continuous, $K$ is closed in $M$, and hence compact.
For every
$x\in K$,
Lemma~\ref{lem:sep0}
provides
a neighborhood
$\mathcal V_x$
of
$f$
in
$\Homeo(M)$
and
a neighborhood
$U_x$
of
$x$
such that
\[
g[B(y,r)]
\supset
B(f(y),\lambda_0r)
\]
for every
$g\in\mathcal V_x$,
every
$y\in U_x$,
and every
$r\in[\delta_0,\bar\de_x]$.
Since
$K$
is compact,
there exist
$x_1,\dots,x_N \in K$
such that
\[
K
\subset
U_{x_1}\cup\cdots\cup U_{x_N}.
\]

Set $\mathcal V
=
\mathcal V_{x_1}
\cap\cdots\cap
\mathcal V_{x_N}.
$
Let
$g\in\mathcal V$
and suppose that
$B(y,r)\subset U$
with
$\delta_0\le r\le\delta$.
Then
$\de_0\le r\le\delta_y$,
so
$y\in K$.
Hence
$y\in U_{x_i}$
for some
$i\in \{1,\ldots,N\}$.
Since
$g\in\mathcal V_{x_i}$,
Lemma~\ref{lem:sep0}
implies that
\[
g[B(y,r)]
\supset
B(f(y),\lambda_0r),
\]
which proves the proposition.
\end{proof}
The proposition above provides a uniform perturbation estimate on every
expanding domain.
In the next subsection, we apply this result to the domain sequences
associated with an expanding datum and establish their uniform
contraction properties.

\subsection{ Uniform contraction of domain sequences}
We now combine the local expansion estimates along codes to obtain
uniform contraction properties for domain sequences. We begin with
their nestedness.

\begin{lemma}\label{lem:insp}
Suppose that $\Gamma<\homM$ is expanding on $M$. Let $\D=(I,\U, S,\la,\de)$ be an expanding datum for the $\Gamma$-action on $M$.
Let $0<\eta\leq\delta$, and let $\alpha$ be a
$(\mathcal D,\eta)$-code for $x\in M$. Then the $\alpha$-domain sequence $(D_k^\alpha)_{k\geq0}$ is nested.
More precisely, for every $k\geq0$,
\begin{itemize}
\item[\textup{(i)}] $D_{k}^{\al}\supset D_{k+1}^\al$,
\item[\textup{(ii)}] $D^\al_k \subset s_{\al(0)} \left[B\left(x_1^\al, \frac{\eta}{\la^k}\right)\right]$.
\end{itemize}
\end{lemma}
\begin{proof}
We first prove~\textup{(i)}. For every $k\ge0$, the map
$s_{\alpha(k+1)}^{-1}$ is
$(\lambda,U_{\alpha(k+1)};\delta)$-expanding, and
\[
B(x_{k+1}^\alpha,\eta)\subset U_{\alpha(k+1)}.
\]
Hence
\[
s_{\alpha(k+1)}^{-1}
\bigl[B(x_{k+1}^\alpha,\eta)\bigr]
\supset
B\bigl(
s_{\alpha(k+1)}^{-1}(x_{k+1}^\alpha),
\lambda\eta
\bigr)
=
B(x_{k+2}^\alpha,\lambda\eta).
\]
In particular,
\[
B(x_{k+2}^\alpha,\eta)
\subset
s_{\alpha(k+1)}^{-1}
\bigl[B(x_{k+1}^\alpha,\eta)\bigr].
\]
Applying the homeomorphism $r^\alpha_{k+1}$ to this inclusion gives
\[
\begin{aligned}
D_{k+1}^\alpha
=
r_{k+1}^\alpha
\bigl[B(x_{k+2}^\alpha,\eta)\bigr]
\subset
r_{k+1}^\alpha
s_{\alpha(k+1)}^{-1}
\bigl[B(x_{k+1}^\alpha,\eta)\bigr]
=
r_k^\alpha
\bigl[B(x_{k+1}^\alpha,\eta)\bigr]
=
D_k^\alpha.
\end{aligned}
\]
Thus $(D_k^\alpha)_{k\ge0}$ is nested.

We now prove~\textup{(ii)}. Since
$s_{\alpha(j)}^{-1}$ is
$(\lambda,U_{\alpha(j)};\delta)$-expanding and $B(s_{\al(j)}(x_{j+1}^\al), \eta)=B(x_{j+1}^\al, \eta)$ is contained in $U_{\al(j)}$,
Lemma \ref{lem:con} gives that 
for all $0<r\le\lambda\eta$,
\[
\begin{aligned}
s_{\alpha(j)}
\bigl[B(x_{j+1}^\alpha,r)\bigr]
\subset
B\bigl(
s_{\alpha(j)}(x_{j+1}^\alpha),
\lambda^{-1}r
\bigr)
=
B(x_j^\alpha,\lambda^{-1}r).
\end{aligned}
\]

We apply this contraction estimate successively along the ray. Since all the radii
$
\eta,\lambda^{-1}\eta,\ldots,\lambda^{-k+1}\eta
$
lie in $(0,\lambda\eta]$ for $k\ge 1$,
\begin{align*}
D^\alpha_k
&= s_{\alpha(0)}\cdots s_{\alpha(k)}[B(x_{k+1}^\alpha,\eta)] \\
&\subset s_{\alpha(0)}\cdots s_{\alpha(k-1)}
   [B(x_{k}^\alpha,\lambda^{-1}\eta)] \\
&\subset s_{\alpha(0)}\cdots s_{\alpha(k-2)}
   [B(x_{k-1}^\alpha,\lambda^{-2}\eta)] \\
&\phantom{\subset}\vdots \hspace{1.8cm} \vdots \\
&\subset s_{\alpha(0)}
   [B(x_{1}^\alpha,\lambda^{-k}\eta)],
\end{align*}
which is the desired estimate.
\end{proof}

The closures $\cl(D_k^\alpha)$ form a decreasing sequence of compact
sets containing $x$.
However, this qualitative property alone is insufficient for our
purposes.
To construct the semi-conjugacy, we must prove that these domains
shrink uniformly, independently of the choice of code.

Recall that $\al(0)$ is chosen arbitrarily. Hence $s_{\al(0)}$ may not be $\la^{-1}$-contracting on $B(x_{1}^\alpha,\lambda^{-k}\eta)$. 
Moreover, no uniform Lipschitz bound for
$s_{\alpha(0)}$ is assumed. Consequently, the estimate in Lemma~3.4 does not directly imply
\[
D_k^\alpha\subset B(x,\lambda^{-k-1}\eta).
\] However, the uniform continuity of $s_{\alpha(0)}$ provides uniform
control of the diameter of $s_{\alpha(0)}[B(x_{1}^\alpha,\lambda^{-k}\eta)]$. The following corollary records the most immediate consequence of
Lemma~\ref{lem:insp}.

\begin{corollary}\label{cor:nsp}
Let $\Gamma<\Homeo(M)$ be expanding on a compact metric space $M$
with an expanding datum
\[
\D=(I,\U,S,\lambda,\delta),
\]
and fix $0<\eta\le\delta$. Then, for every $k\geq 1$, there exists
$n_k\in\N$ such that, for every $x\in M$ and every
$(\D,\eta)$-code $\alpha$ for $x$,
\[
\diam(D_{n_k}^\alpha)\le \frac1k.
\]
Consequently, 
\[
\{x\}=\bigcap_{k=0}^\infty \cl(D^\al_k).
\]
In addition, for every pair of distinct points $x, y\in M$, there is an element $\gamma\in \Gamma$ such that \[d(\gamma(x),\gamma(y))\ge \de.\]
\end{corollary}
\begin{proof}
By the compactness of $M$, every $s_i$ is uniformly continuous. Hence for each $k\in \N$, there exists a constant $\de_{k,i}>0$ such that 
for all $x\in M$, \[ s_i[B(x,\de_{k,i})]\subset B(s_i(x),(2k)^{-1}).\]
Setting $\de_k=\min_{i\in I}\de_{k,i}$, for all $s_i\in S$ and all $x\in M$, the following holds.
\[ s_i[B(x,\de_k)]\subset B(s_i(x),(2k)^{-1}).\]
For each $k\in \N$, choose a natural number $n_k\in \N$ such that $\la^{-n_k}\eta <\de_k$. 
Lemma \ref{lem:insp} and the choice of $\delta_k$ give
\begin{align*}
\diam(D^\al_{n_k}) &\leq \diam(s_{\al(0)}[B(x_1^\al, \la^{-n_k}\eta)]) \\
 &\leq \diam(s_{\al(0)}[B(x_1^\al, \de_k)]) \\
&\leq \diam(B(s_{\al(0)}(x_1^\al), (2k)^{-1})) \\
&= \diam(B(x, (2k)^{-1}))\\
&\leq k^{-1},
\end{align*} 
which proves the first assertion.

Let $x, y \in M$ be distinct. Consider a $(\D,\de)$-code $\al$ for $x$. Since $\diam(D_k^\alpha)$ converges to $0$ and $x\in D_k^\alpha$ for every $k$,
there exists $k_0$ such that 
\[ y\notin D^\al_{k_0} \textup{ and thus } (r_{k_0}^\al)^{-1}(y)\notin B(x^\al_{k_0+1},\de) =  B((r^\al_{k_0})^{-1}(x),\de).\]
Thus we obtain $d((r^\al_{k_0})^{-1}(x),(r^\al_{k_0})^{-1}(y))\geq \de$ as desired.
\end{proof}

The preceding corollary shows that every domain sequence converges
uniformly to its base point. The following proposition gives the uniform inclusion estimate needed
later in the perturbation argument. It is independent of both the
initial point and the chosen code. This estimate is the key dynamical ingredient in the
proof of the local $C^0$-semi-rigidity theorem.

\begin{proposition}[Uniform contraction property]\label{prop:ucp} Let $\Gamma < \operatorname{Homeo}(M)$ be expanding on a compact metric space $M$ with an expanding datum $\D=(I,\U, S,\la,\de)$. Let $0<\eta\le\de$. Then, for every nonempty finite subset $F\subset \homM$, there exists $N\in \N$ such that for every $x\in M$, every $(\D,\eta)$-code $\alpha$ for $x$, and every $f\in F$, \[D_N^\al= r^\alpha_{N} [B(x^\alpha_{N+1},\eta)]\subset f[B(f^{-1}(x),\eta)].\]
The integer $N$ depends only on $\D$, $\eta$, and the finite set $F$,
and is independent of $x$ and $\alpha$.
\end{proposition}
\begin{proof}

Fix $\eta\in(0,\delta]$. We first prove the assertion when
$F=\{f\}$ is a singleton.

Since $M$ is compact, the inverse map $f^{-1}$ is uniformly continuous.
Hence there exists $t>0$ such that, for every $x\in M$,
\[
B(x,t)
\subset
f\bigl[B(f^{-1}(x),\eta)\bigr].
\]

By Corollary~\ref{cor:nsp}, there exists $N_f\in\mathbb N$ such that,
for every $x\in M$ and every $(\mathcal D,\eta)$-code $\alpha$ for $x$,
\[
\operatorname{diam}(D_{N_f}^\alpha)<t.
\]
Since $x\in D_{N_f}^\alpha$, it follows that
\[
D_{N_f}^\alpha
\subset
B(x,t)
\subset
f\bigl[B(f^{-1}(x),\eta)\bigr].
\]
This proves the assertion when $F$ is a singleton.

Now let $F\subset\Homeo(M)$ be finite. For each $f\in F$, let $N_f$
be as above, and set
\[
N=\max_{f\in F}N_f.
\]
Since every domain sequence is nested, we have
\[
D_N^\alpha\subset D_{N_f}^\alpha
\]
for every $f\in F$. Therefore,
\[
D_N^\alpha
\subset
D_{N_f}^\alpha
\subset
f\bigl[B(f^{-1}(x),\eta)\bigr]
\]
for every $f\in F$, every $x\in M$, and every
$(\mathcal D,\eta)$-code $\alpha$ for $x$.
This completes the proof.
\end{proof}

The proposition provides a form of uniform contraction that is
compatible with composition by any prescribed finite family of
homeomorphisms.

\subsection{Uniform Contraction Along Nearby Rays}


We next compare domain sequences arising from different codes. The
following lemma shows that uniformly close points on two rays yield
nested domain sets after a uniformly bounded shift.

\begin{lemma}[Uniform contraction along nearby rays]\label{lem:orucp}
 Let $\Gamma < \operatorname{Homeo}(M)$ be a finitely generated subgroup with an expanding datum $\D=(I,\U, S,\la,\de)$. 
Then, for every $L\geq 0$ and $0<\eta\le \de$, there exists $N\in \mathbb N$ such that if $d_S(r^\al_i, r^\be_j)\leq L$ and $\al, \be$ are $(\D,\eta)$-codes with the same initial point, then $D^\al_{i+N}
\subset 
D^\be_j$.
\end{lemma}
\begin{proof}
Let $\al$ and $\be$ be $(\D,\eta)$-codes for $x\in M$. 
Fix
$L\ge0$,
and let
\[
W_L
=
\{
g\in\Gamma:
|g|_{S}\le L
\}.
\]
Since
$\Gamma$
is finitely generated,
the set
$W_L$
is finite.

Suppose that
$
d_S
(r_i^\alpha,
r_j^\beta)
\le L.
$
Then there exists
$f\in W_L$
such that
$
(r_i^\alpha)^{-1}
r_j^\beta
=
f.
$
By definition, $$f(x^\be_{j+1})=(r^\al_i)^{-1}r^\be_j(x^\be_{j+1})=(r^\al_i)^{-1}(x)=x^\al_{i+1}.$$
Define the shifted code $\gamma$ by
\[
\gamma(n)=\alpha(i+1+n).
\]
Then $\gamma$ is a $(\mathcal D,\eta)$-code for
$x_{i+1}^\alpha$.
Applying Proposition \ref{prop:ucp} to the finite set $W_L$, choose
$N=N(\mathcal D,\eta,L)\geq 1$ such that
 \[ r^{\gamma}_{N-1} [B(x^\gamma_{N},\eta)]\subset f[B(f^{-1}(x^\al_{i+1}),\eta)]=(r^\al_i)^{-1}r^\be_j[B(x^\be_{j+1},\eta)]\]
and thus 
 \[ r^\al_i r^{\gamma}_{N-1} [B(x^\gamma_{N},\eta)]\subset r^\be_j[B(x^\be_{j+1},\eta)].\]
By the definition of the shifted code,
\[
r_i^\alpha r_{N-1}^\gamma=r_{i+N}^\alpha
\quad\text{and}\quad
x_N^\gamma=x_{i+N+1}^\alpha.
\]Therefore,
\[
D_{i+N}^\alpha\subset D_j^\beta,
\]
as required.
\end{proof}

Section ~3 establishes that domain sequences
associated with an expanding datum satisfy strong uniform contraction
properties.
These estimates depend only on the expanding datum and are
independent of the particular choice of code.

\section{Perturbations of Expanding Data}
In the previous section, we obtained uniform contraction estimates for
domain sequences associated with an expanding datum. We now show that
suitable versions of these estimates persist under small $C^0$
perturbations.

The main difficulty is that the local expansion radius is not
stable under perturbations.
To overcome this difficulty, we introduce a global quantity that provides a uniform lower
bound for the expansion radius and behaves continuously with respect to
the $C^0$ topology.

\subsection{A Uniform Expansion Radius}
Let $f\in \homM$ and fix $\eta>0$.
For $x\in M$, define the pointwise expansion radius
\[
R_\eta^f(x)
=
\sup\{r>0 \mid B(x,r)\subset f[B(f^{-1}(x),\eta)]\}.
\]
Uniform continuity of $f^{-1}$ guarantees the existence of a constant $t>0$, independent of $x\in M$, such that 
$$B(x,t)\subset f[B(f^{-1}(x),\eta)],$$ for all $x\in M$ and thus $$ \inf_{x\in M}R_\eta^f(x)>0.$$
One may ask whether $R_\eta^f$ is continuous. The following example shows that $R_\eta^f$ need not be continuous.

\begin{example}[Discontinuity of the pointwise expansion radius]
Let
\[
M=\{0\}\cup \left\{\frac1n  : n\ge 1\right\}\cup\{-2\}\subset\mathbb R
\]
with the induced metric, $f=\mathrm{id}_M$, and $\eta=1$.
Then
\[
R_1^{\mathrm{id}_M}(x)=\sup\{r>0  : B(x,r)\subset B(x,1)\}.
\]

For $x_n=1/n$ with $n\geq2$,
\[
B(x_n,1)=M\setminus\{-2\},
\]
and hence
$
R_1^{\id_M}(x_n)=d(x_n,-2)=2+\frac1n$ converges to $2$.
On the other hand,
\[
B(0,1)=\{0\}\cup\left\{\frac1n : n\geq2\right\},
\]
so \(R_1^{\id_M}(0)=1\). Since \(x_n\) converges to $0$, the function
\(R_1^{\id_M}\) is not continuous at \(0\).
\end{example}

The previous example shows that
the pointwise expansion radius $R_\eta^f$
is generally not continuous.
Therefore it is not suitable for perturbation arguments.
Instead of considering the pointwise expansion radius,
we introduce a uniform expansion radius obtained by taking the largest
radius that works simultaneously at every point.

\begin{proposition}\label{prop:injr}
Let $(M,d)$ be a compact metric space and let
$0<\eta\leq\diam M$.
For each $f\in \homM$, define
\[
T_\eta(f)
=
\sup\Bigl\{
t>0:
B(x,t)\subset f\bigl[B(f^{-1}(x),\eta)\bigr]
\text{ for all }x\in M
\Bigr\}.
\]
Then the map $T_\eta:\homM \to (0,\infty)$ is continuous with respect to  the $C^0$ topology on
$\Homeo(M)$.
\end{proposition}

\begin{proof}
Observe that the defining inclusion admits the following
equivalent reformulation:
\[
B(x,t)\subset f[B(f^{-1}(x),\eta)]
\]
for all $x\in M$ if and only if
$d(a,b)\ge \eta$ implies $d(f(a),f(b))\ge t$.

Assume first that
\[
B(x,t)\subset f[B(f^{-1}(x),\eta)]
\]
for every $x\in M$. If $d(a,b)\geq\eta$, then $b\notin B(a,\eta)$, and hence
$f(b)\notin f[B(a,\eta)]$. The hypothesis $B(f(a),t)\subset f[B(a,\eta)]$ leads us to conclude that $f(b)\notin B(f(a),t)$, i.e., 
$d\bigl(f(a),f(b)\bigr)\ge t$.

Conversely, if the latter condition holds and
$y\in B(x,t)$, then $d\bigl(f^{-1}(x),f^{-1}(y)\bigr)<\eta$, which implies
\[
y\in f\bigl[B(f^{-1}(x),\eta)\bigr].
\]
Therefore, the two conditions are equivalent.

From the latter condition, we reformulate $T_\eta(f)$ by 
\[
T_\eta(f)
=
\min\Bigl\{
d(f(a),f(b))
:
d(a,b)\ge \eta
\Bigr\}
=\min_{(a,b)\in A_\eta}d(f(a),f(b))
\]
where 
$A_\eta
=
\{(a,b)\in M\times M:d(a,b)\ge \eta\}$.
Since $M$ is compact, $A_\eta$ is also compact.


Recall the metric $d_\infty$ on $\Homeo(M)$ defined in Section~3.1.
 For every $(a,b)\in A_\eta$, the triangle inequality gives
\[
\bigl|
d(g(a),g(b))
-
d(f(a),f(b))
\bigr|
\le
d(g(a),f(a))
+
d(g(b),f(b)) \le
2d_\infty(g,f).
\]
Since $T_\eta(f)\le d(f(a),f(b))$ for all $(a,b)\in A_\eta$, it follows that for all $(a,b)\in A_\eta$,
\[ T_\eta(f)-2d_\infty(g,f) \le d(f(a),f(b))-2d_\infty(g,f)\le d(g(a),g(b)),\]
which gives rise to $T_\eta(f)-2d_\infty(g,f) \le T_\eta(g)$. Similarly, we obtain
\[ T_\eta(g)\le d(g(a),g(b)) \le d(f(a),f(b))+2d_\infty(g,f) \] and then
$ T_\eta(g)\le T_\eta(f)+2d_\infty(g,f)$.
Finally, we get 
\[
|T_\eta(g)-T_\eta(f)|
\le
2d_\infty(g,f),
\]
which implies that $T_\eta$ is $2$-Lipschitz with respect to $d_\infty$.
\end{proof}

\subsection{Perturbed Domain Sequences}

We return to the expanding datum $
\mathcal D=(I,\mathcal U,S,\lambda,\delta).
$
The nestedness and uniform contraction properties established in
Section~3 are consequences of the local expansion property of the
generators.

Let $\rho : \Gamma\to\Homeo(M)$ be a homomorphism. For a
$(\mathcal D,\eta)$-code $\alpha$ for $x\in M$, define the perturbed
domain sequence by
\[
D_k^{\alpha,\rho}
:=
\rho(r_k^\alpha)[B(x_{k+1}^\alpha,\eta)].
\]
The following proposition shows that suitably weakened versions of these
properties remain valid after sufficiently small $C^0$ perturbations.



\begin{proposition}\label{prop:perturbp2}
Suppose that $\Gamma < \operatorname{Homeo}(M)$ is expanding on a compact metric space $M$ with an expanding datum $\D=(I,\U,S,\la,\de)$. For each $0<\de_0<\de$ and each $L\geq 0$, there exists a neighborhood $\mathcal V$ of $\iota:\Gamma\to \homM$ such that the following holds.
\begin{itemize}
\item[\textup{(i)}] $D_{k}^{\al,\rho}\supset D_{k+1}^{\al,\rho}$ for all $k\geq0$, all $\rho\in \mathcal V$, all $\eta\in [\de_0,\de]$ and all $(\D,\eta)$-code $\al$.
\item[\textup{(ii)}] There exists $n_0=n_0(\mathcal D,\delta_0,L)\in\mathbb N$ such that for all $\eta\in [\de_0,\de]$ and all $(\D,\eta)$-codes $\al, \be$ with $x^\al_0=x^\be_0$, 
if $d_S(r^\al_i, r^\be_j)\leq L$, then $$D^{\al,\rho}_{i+n_0} \subset D^{\be,\rho}_j$$ for all $\rho\in \mathcal V$. 
\item[\textup{(iii)}] If $\al$ is a special code for $x$ with $B(x,\eta)\subset U_{\al(0)}$, then 
$$D^{\alpha,\rho}_{n_0} \subset B\left(x,\frac{\de_0}{3}\right).$$
\end{itemize}
\end{proposition}
\begin{proof}
We first prove (i). Let $1<\la_0<\la$ and $0<\de_0<\de$. For each $i\in I$, Proposition \ref{prop:deformp} gives a neighborhood
$\mathcal W_i$ of $s_i^{-1}$ in $\Homeo(M)$. Since $I$ is finite,
there exists a neighborhood $\mathcal V_0$ of $\iota$ in
$\Hom(\Gamma,\Homeo(M))$ such that
\[
\rho(s_i)^{-1}\in\mathcal W_i
\]
for every $\rho\in\mathcal V_0$ and every $i\in I$.
Then for all integers $k\geq 0$, all $\eta\in [\de_0,\de]$, and all $(\D,\eta)$-code $\alpha$, we have
\[ \rho(s_{\al(k+1)})^{-1}[B(x^\al_{k+1},\eta)]\supset B(s_{\al(k+1)}^{-1}(x_{k+1}^\al), \lambda_0 \eta)=B(x^\al_{k+2},\la_0 \eta)\supset B(x^\al_{k+2}, \eta).\]
Applying $\rho(r_{k+1}^\alpha)$ to this inclusion gives $D_{k}^{\al,\rho}\supset D_{k+1}^{\al,\rho}$.

To prove (ii) and (iii), let $$W_L=\{g\in\Gamma : |g|_S\leq L\} \text{ and }t_L=\min \{T_{\de_0/2}(f) :  f\in W_L\}>0$$ where $T_{\de_0/2}$ is the function defined in Proposition \ref{prop:injr}. Since $T_{\de_0/2}$ is continuous and $W_L$ is finite, there exists a neighborhood $\mathcal V_1$ of $\iota:\Gamma\to \homM$ such that for all $\rho\in \mathcal V_1$,
\[ \min \{T_{\de_0/2}(\rho(f)) : f\in W_L\} > t_L/2.\]
In other words, for all $f\in W_L$ and all $x\in M$, 
\[ B(x,t_L/2)\subset \rho(f)[B(\rho(f)^{-1}(x),\de_0/2)].\]
Since $W_L$ and $S$ are finite, there exists a neighborhood $\mathcal V_2$ of $\iota$ such that
\[
d_\infty(f, \rho(f))
<\frac{\delta_0}{2}\leq \frac{\eta}{2} \text{ and } d_\infty(\rho(s),s)<\frac{t_L}{4}
\]
for every $f\in W_L$ and $s\in S$.
Then for all $\eta\in [\de_0,\de]$ and  all $f\in W_L$, it is easy to check that $$B(\rho(f)^{-1}(x),\eta/2)\subset B(f^{-1}(x),\eta).$$

Choose $k_0\in \mathbb N$ so that $1/k_0<t_L/4$. Due to the continuity of $T_{1/k_0}$, there exist a neighborhood $\mathcal V_3$ of $\iota:\Gamma\to \homM$ and $\epsilon_0>0$ such that for all $\rho \in \mathcal V_3$, all $x\in M$ and all $s_i\in S$,
\[ \rho(s_i) [B(x,\epsilon_0)]\subset B(\rho(s_i)(x),1/{k_0}).\]
Now choose $n_0\in \mathbb N$ such that \[ \la_0^{-n_0}\de<\epsilon_0 \text{ and } \la_0^{-n_0-1}\de<\frac{\de_0}{3}.\]
Proposition \ref{prop:deformp} gives a neighborhood $\mathcal V_4$ of $\iota$ such that 
 for every $\rho\in\mathcal V_4$, every $i\in \I$ and every $r\in[\la_0^{-n_0-1}\de_0,\de]$,
\[ \rho(s_i)^{-1}[B(x,r)]\supset B(s_i^{-1}(x),\la_0r) \textup{ whenever } B(x,r)\subset U_i.\]
 By a similar argument as in the proof of Lemma \ref{lem:insp},
\begin{align*}
D^{\alpha, \rho}_{n_0}
&= \rho(s_{\alpha(0)})\cdots \rho(s_{\alpha(n_0)})[B(x_{n_0+1}^\alpha,\eta)] \\
&\subset \rho(s_{\alpha(0)})\cdots \rho(s_{\alpha(n_0-1)})
   [B(x_{n_0}^\alpha,\lambda_0^{-1}\eta)] \\
&\subset \rho(s_{\alpha(0)})\cdots \rho(s_{\alpha(n_0-2)})
   [B(x_{n_0-1}^\alpha,\lambda_0^{-2}\eta)] \\
&\phantom{\subset}\vdots \hspace{1.8cm} \vdots \\
&\subset \rho(s_{\alpha(0)})
   [B(x_{1}^\alpha,\lambda_0^{-n_0}\eta)].
\end{align*}
for every $(\D,\eta)$-codes $\alpha$ and every $\eta\in [\de_0,\de]$. If $\al$ is a special code for $x$, we have 
\begin{align*}
D^{\alpha,\rho}_{n_0} \subset \rho(s_{\alpha(0)}) [B(x_{1}^\alpha,\lambda_0^{-n}\eta)] \subset B(x,\lambda_0^{-n_0-1}\eta)\subset B(x,\de_0/3).
\end{align*}

Set $\mathcal V=\mathcal V_0\cap \mathcal V_1\cap \mathcal V_2\cap \mathcal V_3 \cap \mathcal V_4$. 
Then for every $\rho\in\mathcal V$, every $(\D,\eta)$-code with $\eta\in [\de_0,\de]$ and every $f\in W_L$,
\begin{align*}
D^{\alpha,\rho}_{n_0} &\subset \rho(s_{\alpha(0)})[B(x_{1}^\alpha,\lambda_0^{-n_0}\eta)] \\
& \subset  \rho(s_{\alpha(0)})[B(x_{1}^\alpha,\epsilon_0)] \\
&\subset B(\rho(s_{\al(0)})(x_1^\al),1/k_0) \\
&\subset B(\rho(s_{\al(0)})(x_1^\al),t_L/4) \\
&\subset B(s_{\al(0)}(x_1^\al),t_L/2) \\
&= B(x,t_L/2) \\
&\subset \rho(f)[B(\rho(f)^{-1}(x),\de_0/2)]\\
&\subset \rho(f)[B(f^{-1}(x),\de_0)]\\
&\subset \rho(f)[B(f^{-1}(x),\eta)].
\end{align*}
which proves the second assertion by the same argument as in
Lemma \ref{lem:orucp}.
\end{proof}

Unlike the unperturbed case,
Proposition~\ref{prop:perturbp2}
does not imply that the diameters of the perturbed domain sequences
converge to zero. Thus the perturbed domains need not shrink to points, even though they
remain nested and satisfy the comparison property above.


\section{Local $C^0$-semirigidity}

We are now ready to prove the main theorem of the paper.
The perturbation results established in the previous section provide
nested families of perturbed domains satisfying the same compatibility
properties as in the unperturbed setting.
Using the meandering hyperbolicity assumption, we show that these
families determine a $\rho$-equivariant partition of $M$, from which
the desired semi-conjugacy is obtained.

\begin{definition}[Local $C^0$-semi-rigidity]\label{def:semirigid}
Let
$
\rho_0 : \Gamma\to\Homeo(M)
$
be an action. Then $\rho_0$ is \emph{locally semi-rigid} in the $C^0$ topology if there is a neighborhood $\mathcal V$ of $\rho_0$ in $\Hom(\Gamma,\homM)$ such that for each $\rho \in \mathcal V$, there exists a continuous surjective map
$\phi :  M\to M$ such that
\[
\phi\circ\rho(g)=\rho_0(g)\circ \phi
\]
for every $g\in\Gamma$. Such a map $\phi$ is called \emph{a semi-conjugacy}, or more precisely, a $(\rho, \rho_0)$-semi-conjugacy.
\end{definition}

\begin{proof}[Proof of Theorem \ref{thm:main}]
Let $(\D\prec\D';L)$ be a meandering-hyperbolicity datum for the $\Gamma$-action on $M$.
Write 
\[
\mathcal D=(I,\mathcal U,S,\lambda,\delta)
\quad\text{and}\quad
\mathcal D'=(I',\mathcal U',S',\lambda',\delta').
\]

Fix $0<\delta_0<\delta$.  Apply Proposition \ref{prop:perturbp2} to the refined datum $\mathcal D'$ and the
constant $L$, and let $\mathcal V$ be the resulting neighborhood.
Since every $\mathcal D$-code is also a $\mathcal D'$-code, the
conclusions apply to all codes used below.
Throughout the proof, let
$\rho\in\mathcal V$ and fix $\eta\in[\delta_0,\delta]$.
We divide the proof into several steps.

\medskip

\noindent
\textbf{Step 1. Construction of the perturbed limit sets.}

\medskip

For each
$x\in M$,
choose an arbitrary
$(\mathcal D,\eta)$-code
$\alpha$
for
$x$
and define

\[
\Phi_\alpha(x)
=
\bigcap_{k=0}^{\infty}
\cl({D_k^{\alpha,\rho}}).
\]

By Proposition \ref{prop:perturbp2}(i), the compact sets
$
\cl(D_k^{\alpha,\rho})
$
form a nested sequence. Since each is nonempty and $M$ is compact,
their intersection $\Phi_\alpha(x)$ is nonempty.

\medskip

\noindent
\textbf{Step 2. Independence of the choice of code.}

\medskip

We next show that the above definition is independent of the chosen
$(\mathcal D,\eta)$-code.

\begin{lemma}
For every $\eta\in[\delta_0,\delta]$, every $x\in M$, every pair of
$(\mathcal D,\eta)$-codes $\alpha,\beta$ for $x$, and every
$\rho\in\mathcal V$, one has
\[ \Phi_\alpha(x)=\bigcap_{k=0}^\infty \cl(D^{\al,\rho}_k)=\bigcap_{k=0}^\infty \cl(D^{\beta,\rho}_k)=\Phi_\beta(x). \]
\end{lemma}
\begin{proof}

Due to the meandering hyperbolicity of the $\Gamma$-action on $M$, any two $(\D,\eta)$-codes $\al$ and $\beta$ for $x$ can be regarded as $(\mathcal D',\eta)$-codes and moreover, there exists a finite chain of interpolating $(\D',\eta)$-codes $\al=\gamma_0$, $\gamma_1,\ldots, \gamma_n=\beta$ for $x$ such that 
\[
r^{\ga_0} \approx_L r^{\ga_1} \approx_L \cdots \approx_L r^{\ga_n}. 
\]
It therefore suffices to consider the case
$r^\alpha\approx_Lr^\beta$. If $r^\alpha\approx_Lr^\beta$, there exist stictly increasing sequences $(a_i)$ and $(b_i)$ in $\mathbb N$ such that 
\[ d(r^\al_{a_i},r^\beta_{b_i})\leq L.\] 
By Proposition \ref{prop:perturbp2} (ii), 
\[D^{\al,\rho}_{a_i+n_0}\subset D^{\beta,\rho}_{b_i} \textup{ and } D^{\beta,\rho}_{b_i+n_0}\subset D^{\al,\rho}_{a_i}.\]
Since $a_i, b_i$ go to the infinity and both perturbed domain sequences are
nested, the first family of inclusions implies
\[
\bigcap_{k=0}^\infty \cl(D_k^{\alpha,\rho})
\subset
\bigcap_{k=0}^\infty\cl(D_k^{\beta,\rho})
\]
while the second gives the reverse inclusion.
\end{proof}

We may therefore define \[ \Phi(x):=\Phi_\alpha(x), \] independently of the chosen code.

\medskip

\noindent
\textbf{Step 3. Equivariance of the limit sets.}

\medskip

We next show that the family
\(
\{\Phi(x)\}_{x\in M}
\)
is equivariant under the perturbed action.

\begin{lemma}\label{lem:equi}
For every $x\in M$ and every $s\in S$, $$\Phi(s(x))=\rho(s)[\Phi(x)].$$
\end{lemma}
\begin{proof}
Since the domain condition in the definition of a $(\mathcal D,\eta)$-code is imposed only for indices $k\geq1$, the initial symbol may be chosen arbitrarily. Due to the symmetry of $S$, we may choose a $(\mathcal D,\eta)$-code $\alpha$ for $x$ such that \[ s_{\alpha(0)}=s^{-1}. \] Then \[ x_1^\alpha = s_{\alpha(0)}^{-1}(x) = s(x). \] Define the shifted sequence $\alpha_1\colon\mathbb N\to I$ by \[ \alpha_1(k)=\alpha(k+1). \] It is a $(\mathcal D,\eta)$-code for $s(x)$. Since $\rho(s)$ is a homeomorphism, it commutes with closures and intersections, and hence \[ \begin{aligned} \Phi(x) &= \bigcap_{k=0}^{\infty} \cl\left( \rho(s_{\alpha(0)})\cdots\rho(s_{\alpha(k)}) [B(x_{k+1}^\alpha,\eta)] \right)\\ &= \rho(s)^{-1} \bigcap_{k=0}^{\infty} \cl\left( \rho(s_{\alpha(1)})\cdots\rho(s_{\alpha(k)}) [B(x_{k+1}^\alpha,\eta)] \right)\\ &= \rho(s)^{-1}\Phi(sx). \end{aligned} \] Therefore, $\Phi(sx)=\rho(s)[\Phi(x)]. $
\end{proof}

Since $S$ generates $\Gamma$, it follows that
\[
\Phi(\gamma x)=\rho(\gamma)[\Phi(x)]
\]
for every $\gamma\in\Gamma$.

\medskip

\noindent
\textbf{Step 4. The family
$\{\Phi(x)\}_{x\in M}$
forms a partition of
$M$.}

\medskip

The next lemma shows that the limit sets are pairwise disjoint and
cover the whole space.

\begin{lemma}
The collection $\{\Phi(x)\}_{x\in M}$ is a partition of $M$. That is, the union of all elements of the collection is $M$ and any two distinct elements of the collection are disjoint.
\end{lemma}
\begin{proof}
Let $x,y\in M$ be distinct. By Corollary~\ref{cor:nsp}, there exists $\gamma\in\Gamma$ such that
\[
d(\gamma (x),\gamma (y))\ge\delta.
\]
Since the family $\Phi$ is equivariant and $\rho(\gamma)$ is a
homeomorphism of $M$, it suffices to prove disjointness for
$\Phi(\gamma x)$ and $\Phi(\gamma y)$. Thus, after replacing
$(x,y)$ by $(\gamma (x),\gamma (y))$, we may assume that
$d(x,y)\ge\delta$.

Choose special codes $\alpha$ and $\beta$ for $x$ and $y$, respectively. Proposition~4.3(iii) gives \[ \Phi(x) \subset \cl\left({B}\left(x,\frac{\delta_0}{3}\right)\right), \qquad \Phi(y) \subset \cl\left({B}\left(y,\frac{\delta_0}{3}\right)\right). \] Since $\delta_0<\delta$, these two closed balls are disjoint. Hence \[ \Phi(x)\cap\Phi(y)=\varnothing. \]

To verify that $\bigcup_{x\in M}\Phi(x)=M$, it suffices to show that for any $y\in M$, there exists a point $x\in M$ such that $y\in \Phi(x)$.
Fix $y\in M$. Shrink the perturbation neighborhood, if necessary, so that $d_\infty(\rho(s_i),s_i)<(\la-1)\delta/3$ for all $i\in \I$. We also assume that $0<\de_0<\de/3$.

Set $y_0=y$ and choose $\al(0) \in \I$ so that $B(y_0,2\de/3)\subset U_{\al(0)}$ and let $y_1=\rho(s_{\al(0)})^{-1}(y_0)$.
Clearly, we have $$s_{\al(0)}^{-1}[B(y_0,\de/3)]\supset B(s_{\al(0)}^{-1}(y_0),\la\de/3)\supset B(\rho(s_{\al(0)})^{-1}(y_0), \de/3)=B(y_1, \de/3).$$
Then choose $\al(1) \in \I$ so that $B(y_1,2\de/3)\subset U_{\al(1)}$ and let $y_2=\rho(s_{\al(1)})^{-1}(y_1)$. Similarly, 
$$s_{\al(1)}^{-1}[B(y_1,\de/3)]\supset B(s_{\al(1)}^{-1}(y_1),\la\de/3)\supset B(\rho(s_{\al(1)})^{-1}(y_1), \de/3)=B(y_2, \de/3).$$
In this way, we can define a sequence $\al :\N\to \I$ so that $y_{k+1}=\rho(s_{\al(k)})^{-1}(y_k)$ and $B(y_k,\de/3)\subset U_{\al(k)}$ and 
$s_{\al(k)}^{-1}\left[B(y_k,\de/3)\right]\supset B(y_{k+1},\de/3)$. Then we have a nested sequence of domains $(r^\al_k[B(y_{k+1},\de/3)]_{k=0}^\infty)$ and hence there is a point $$x\in \bigcap_{k=0}^\infty   \cl\big(r^\al_k[B(y_{k+1},\de/3)]\big).$$
From the definition of $x$, for any $k\in \N$, $$d((r^\al_k)^{-1}(x),y_{k+1})=d((r^\al_k)^{-1}(x),\rho(r^\al_k)^{-1}(y))\le\de/3$$
and moreover, $$B((r^\al_k)^{-1}(x),\de/3)\subset B(\rho(r^\al_k)^{-1}(y),2\de/3)\subset U_{\al(k+1)}.$$
These imply that $\al$ is a $(\D,\de/3)$-code for $x$ and since $\rho(r^\al_k)^{-1}(y)\in \cl(B(x_{k+1}^\al,\de/3))$ for every $k\geq 0$,
$$y\in \bigcap_{k=0}^\infty   \cl\big(\rho(r^\al_k)[B(x_{k+1}^\al,\de/3)]\big)=\Phi(x).$$
This completes the proof.
\end{proof}

Hence
\(
\{\Phi(x)\}_{x\in M}
\)
is a partition of
\(M\).

\medskip

\noindent
\textbf{Step 5. Construction of the semi-conjugacy.}

\medskip

Since the sets $\Phi(x)$ form a partition of $M$, every $y\in M$ belongs to a unique set $\Phi(x)$. Define \[ \phi(y)=x \text{ whenever }y\in\Phi(x). \] This gives a well-defined map $\phi : M \to M$. Since $\Phi(x)\neq\varnothing$ for every $x\in M$, the map $\phi$ is surjective. Equivariance of the family $\Phi$ gives \[ \phi\circ\rho(\gamma) = \iota(\gamma)\circ\phi \] for every $\gamma\in\Gamma$. We first record the uniform closeness of $\phi$ to the identity. For each $x\in M$, choose a special code $\alpha$ for $x$. By Proposition \ref{prop:perturbp2}(iii), \[ \Phi(x) \subset \cl\left({B}\left(x,\frac{\delta_0}{3}\right)\right). \] Thus, if $y\in\Phi(x)$, then \[ d(y,\phi(y)) = d(y,x) \leq\frac{\delta_0}{3}. \]

The remaining lemma establishes its continuity.

\begin{lemma}
The map $\phi:M\to M$ is continuous.
\end{lemma}
\begin{proof}
Suppose that $(x_n)$ converges to $x$.
Assume, to the contrary, that
$\phi(x_n)$ does not converge to $\phi(x)$.
By compactness, after passing to a subsequence we may assume that
$\phi(x_n)$ converges to  $y$
for some
$y\neq\phi(x)$.

By Corollary~3.5, there exists $\gamma\in\Gamma$ such that \[ d(\gamma (\phi(x)),\gamma (y))\geq\delta. \]
Replacing $x_n$, $x$, and $y$ by $\rho(\gamma)(x_n)$, $\rho(\gamma)(x)$, $\gamma(y)$ respectively, and using equivariance of $\phi$, we may assume that \[ d(\phi(x),y)\geq\delta. \] Choose $0<\delta_0<\delta/3$. By Proposition~\ref{prop:perturbp2},
the neighborhood
$\mathcal V$
may be chosen so that
$$
d(p,\phi(p))
<\frac{\delta_0}{3}
=\frac{\delta}{9}
$$
for every
$p\in M$.
Since
$x_n$ converges to  $x$,
we have, for all sufficiently large $n$,
\[
\begin{aligned}
d(\phi(x_n),\phi(x))
&\le
d(\phi(x_n),x_n)
+d(x_n,x)
+d(x,\phi(x))
\\
&<
\frac{\delta}{9}
+
\frac{\delta}{9}
+
\frac{\delta}{9}
=
\frac{\delta}{3}.
\end{aligned}
\]
Therefore,
\[
d(\phi(x_n),y)
\ge
d(\phi(x),y)-d(\phi(x_n),\phi(x))
>
\delta-\frac{\delta}{3}
=
\frac{2\delta}{3}
\]
for all sufficiently large $n$. 
This contradicts the assumption that
$\phi(x_n)$ converges to $y$.
Hence
$\phi$
is continuous.
\end{proof}

By Step~3,
$
\phi\circ\rho(\gamma)
=
\iota(\gamma)\circ\phi
$
for every
\(
\gamma\in\Gamma.
\)
Step~4 shows that $\phi$ is defined on all of $M$, while the
nonemptiness of every $\Phi(x)$ implies that $\phi$ is surjective.
Therefore
\(
\phi
\)
is a semi-conjugacy from
\(
\rho
\)
to the original action $\iota$.
Finally, for any $\varepsilon>0$, Choose
\[
0<\delta_0<\min\{3\varepsilon,\delta\}.
\] According to Proposition \ref{prop:perturbp2}, there exists a neighborhood $\mathcal V$ of $\iota:\Gamma \to \homM$ such that for all $\rho \in \mathcal V$,
\[ d(p,\phi(p))<\frac{\de_0}{3}\le\varepsilon.\]
This
shows that $\phi$
is chosen close to the identityif $\rho$ is sufficiently close to the inclusion representation
$\iota$.
This completes the proof. 
\end{proof}

Theorem 5.2 provides a unified proof of the previously known
local $C^0$-semi-rigidity theorems for boundary actions of
word-hyperbolic groups and cocompact lattices.
The following corollary records these applications.

\begin{corollary}
The following actions are locally semi-rigid in the $C^0$ topology:
\begin{enumerate}
\item[\textup{(i)}] the action of a word-hyperbolic group on its Gromov boundary;
\item[\textup{(ii)}] the action of a cocompact lattice in a semisimple Lie group
$G$ on the flag variety $G/P$ where $P<G$ is a parabolic subgroup.
\end{enumerate}
\end{corollary}

\begin{proof}
By \cite{Sul}, the boundary action of a word-hyperbolic group is hyperbolic and, in particular, meandering hyperbolic. By \cite{KKL}, the action of a cocompact lattice in a semisimple Lie group on the corresponding flag variety is also meandering hyperbolic. The corollary therefore follows immediately from Theorem \ref{thm:main}. 
\end{proof}

\section*{Acknowledgements}

The author is deeply grateful to Misha Kapovich for many inspiring
discussions and valuable suggestions that significantly influenced
the development of this work.

\vspace{1\baselineskip}\noindent
Department of Mathematics,
Jeju National University,
Jeju 63243,
Republic of Korea\\
\url{sungwoon@jejunu.ac.kr}\\




\end{document}